\documentclass[a4paper,10pt]{article}
\usepackage{amsmath}
\usepackage{amsthm}
\usepackage{amssymb}
\newtheorem{thm}{\bf Theorem}

\newtheorem{prop}[thm]{\bf Proposition}
\newtheorem{cor}[thm]{\bf Corollary}
\newtheorem{dfn}[thm]{\bf Definition}
\newtheorem{conj}[thm]{\bf Conjecture}

\title{On the finite geometry of $W(23,16)$}
\date{November 2011}
\author{Assaf Goldberger}

\begin{document}

\begin{abstract}
We study the local geometry of the zero pattern of a weighing matrix
$W(23,16)$. The geometry consists of $23$ lines and $23$ points where each line
contains $7$ points. The incidence rules are that every two lines intersect in
an odd number of points, and the dual statement holds as well. We show that
more than $50\%$ of the pairs of lines must intersect at a single point, and
construct a regular weighted graph out of this geometry. This might indicate that
a weighing matrix $W(23,16)$ does not exist.

\end{abstract}

\maketitle

A \emph{weighing matrix} of size $n$ and weight $k$, generally denoted $W(n,k)$, is an orthogonal $n\times n$ $\{0,1,-1\}$-matrix with rows of length $\sqrt{k}$. Weighing matrices have applications in Chemistry, Spectroscopy, Quantum Computing and Code Theory. The main mathematical interest is to find or prove inexsitence of a $W(n,k)$. For refences about weighing matrices, see for example, \cite{ref1} or \cite{ref2}. To date, the smallest weighing matrix whose existence is unknown (see \cite{ref3}) is $W(23,16)$. In this note we study the underlying finite geometry of a $W(23,16)$ and conjecture that it does not exist.\\

Suppose that $W$ is a $W(23,16)$ matrix. We define a graph structure on the rows of $W$ by assigning an edge connecting the rows $W_i$ and $W_j$ if the zeros of $W_i$ overlap those of $W_j$ at only one position. We will show that there is at least one such edge, and in fact more than 50\% of the pairs  $W_i$ and $W_j$ are connected.\\

We begin by looking at a set $S$ of cardinality $23$, and a collection of subsets (called \emph{lines}) $\mathcal G(W) := \{L_1,L_2,\ldots,L_{23}\}$, such that $L_i$ is the characteristic set of the zeros of $W_i$. Then (i) each $L_i$ has cardinality $7$, and by the orthogonality relations (ii) $|L_i\cap L_j|=1,3,5,7$ for all $i,j$. The fact that $W^T$ is also of type $W(23,16)$ implies the duality statment: By interchanging between lines and points, any statment that we can prove in general from (i) and (ii)  for $\mathcal G(W)$, is also true in $\mathcal G(W^T)$. For example, for any two points in $S$ there is an odd number of lines (i.e. $1,3,5$ or $7$) passing through them. We construct a graph $\Gamma(W)$ on the index set $\{1,2,\ldots,23\}$ by connecting $i$ and $j$ is $|L_i\cap L_j|=1$. We claim
\begin{prop}
$\Gamma(W)$ must contain at least one edge.
\end{prop}

\begin{proof}
We consider the correspondence relation between unordered pairs of points  $\{i,j\}$ $i\neq j$ and lines $L_k$ containing them. More precisely, the correspondence is the set $$\mathcal C := \{ (\{i,j\},L_k) \ | i\neq j \text{ and } i,j \in L_k \}$$.
We shall count the elements of $\mathcal C$ in two ways. First we sum over pairs $\{i,j\}$ where $\sigma(i,j)\in \{1,3,5,7\}$ is the number of lines $L_k$ containig this pair. Second, we can sum over lines $L_k$, where for each $L_k$ we count the $\pi_k$ number of pairs $\{i,j\}$ contained in it. Thus we have the identity
$$ \sum_{k=1}^{23} \pi_k = \sum_{1\le i<j \le 23} \sigma(i,j).$$
Now, $\pi_k=\binom{7}{2}$ so the left hand side equals $23\cdot 7\cdot 3$. Suppose now that $\sigma(i,j)>1$ for all $\{i,j\}$. Then in fact $\sigma(i,j)\ge 3$ and we have that
$$ 23\cdot 7\cdot 3 = \sum_{k=1}^{23} \pi_k = \sum_{1\le i<j \le 23} \sigma(i,j) \ge 3 \cdot \binom{23}{2} =23\cdot 11 \cdot 3.$$ This is a contradiction, leading us to the conclusion that $\sigma(i,j)=1$ for at least one pair. By duality, there are two lines $L_i,L_j$ with $|L_i\cap L_j|=1$. This is an edge of the graph, and we are done.

\end{proof}

In fact we can conclude a much stronger statment, namely

\begin{prop}
There are at least $138$ edges in $\Gamma(W)$.
\end{prop}
\begin{proof}
This follows from the same identity, that $23\cdot 7\cdot 3 = \sum_{1\le i<j \le 23} \sigma(i,j)$. If $e$ is the number of (dual) edges, then $23\cdot 7\cdot 3 \ge e +3(\binom{23}{2}-e)$, which implies our statement.
\end{proof}
Notice that there can be at most $\binom{23}{2}=253$. This means that there is more that $1/2$ probability that two edges will be connected!\\

In fact we can prove more.

\begin{prop} \label{prop3}
Each vertex of the graph has at least 12 neighbors.
\end{prop}

\begin{proof}
Consider a vertex, say $L_1$. We consider the correspondence
$$\mathcal C(L_1) := \{ (i,L_j) \ | \ i\in L_1\cap L_j \text{ and } j>1\}.$$
For each $i$ let $f_i$ count the number of set containing $i$ besides $L_1$. For each $j$ let $\psi_j=|L_1\cap L_j|.$ Then
$$ \sum_i f_i = \sum_j \psi_j.$$
Let $e$ be the number of $j$ such that $\psi_j=1$. Since $f_i=6$ for all $i$, then $6\cdot 7 \ge e+3(22-e)$ which implies that $e\ge 12$. Notice that $e$ is the number of neighbors of $L_1$.
\end{proof}

\begin{cor}
The graph $\Gamma(W)$ contains a triangle at any vertex of the graph.
\end{cor}

\begin{proof}
Take a vertex $v$ and twelve neighbors $v_1,\ldots,v_{12}$. The vertex $v_1$ has $12$ neighbors (at least), so some of them come from $v_2,\ldots,v_{12}$. Hence we get a triangle.
\end{proof}

\begin{cor}
The graph $\Gamma(W)$ has diameter $2$ (at most).
\end{cor}

\begin{proof}
Taking the neighbors of any two non-neighboring vertices, there must be a common neighbor.
\end{proof}

We now introduce a weighted graph $\widetilde\Gamma(W)$ which in a sense may serve as the complement graph of $\Gamma(W)$.

\begin{dfn}
The graph $\widetilde \Gamma(W)$ is a weighted graph whose vertex set is the set $\{1,2,\ldots,23\}$ end we connect $i$ and $j$ with an edge of weight $1,2,3$ if $|\ell_i\cap \ell_j|=3,5,7$ respectively.
\end{dfn}

\begin{prop}
$\widetilde \Gamma(W)$ is a regular weighted graph of degree $10$.
\end{prop}p

\begin{proof}
The proof follows from examining more carefully the proof of Proposition \ref{prop3}. Again we take a vertex, say $1$, and study the correspondence
$$\mathcal C(L_1) := \{ (i,L_j) \ | \ i\in L_1\cap L_j \text{ and } j>1\}.$$
For each $i$ let $f_i$ count the number of lines containing $i$ except for $L_1$. For each $j$ let $\psi_j=|L_1\cap L_j|.$ Then we have
\begin{equation}\label{eq1} \sum_i f_i = \sum_j \psi_j.\end{equation} Let $n_k$ be the number of lines $L_j$ that intersect $L_1$ with cardinality $k$. Then the right hand side of \eqref{eq1} is rewritten as $\sum_k kn_k = n_1+3n_3+5n_5+7n_7$. The left hand side equals $6\cdot 7=42$ since $f_i=6$ for all $i\in L_1$. it follows that
$$n_1+3n_3+5n_5+7n_7=42.$$
However, $n_1+n_3+n_5+n_7=22$ as there are $23$ lines. Combining these two facts together yields 
$$2n_3+4n_5+6n_7=20 \implies n_3+2n_5+3n_7=10.$$
The proof is finished.
\end{proof}

\begin{conj}
A weighing matrix $W(23,16)$ does not exist.
\end{conj}

We base our conjecture on the fact that (1) The underlying set of the geometry has only $23$ points, (2) Most pairs of lines have a single intersection point, and (3) there are many triples with pairwise intersection of a single point. Each triple covers at   least $18$ points. It seems unlikey that all this can be packed in a small set of size $23$.


\begin{thebibliography}{1}

\bibitem{ref1}  I. S. Kotsireas, C. Koukouvinos, J. Seberry , New weighing matrices constructed from two circulant submatrices,
{\em Optimizations Letters}, January 2012, Volume 6, Issue 1, pp 211-217

\bibitem{ref2} C. Koukouvinos and J. Seberry, Weighing matrices and their applications, {\em JSPI}, 62 (1997) 91-101.

\bibitem{ref3} C. J. Colbourn and  J. H. Dinitz, {\em Handbook of Combinatorial Designs, Second Edition}, Taylor and Francis 2006,ISBN-13: 978-1584885061

\end{thebibliography}
\end{document}